\documentclass[12pt]{amsart}
\usepackage{amsmath}
\usepackage{amsfonts, amssymb, euscript, mathrsfs}
\usepackage{amsthm, upref}
\usepackage{graphicx}
\usepackage[usenames, dvipsnames]{color}

\usepackage{tikz}
\usepackage{enumerate}
\usepackage[tmargin=0.8in,bmargin=0.8in,lmargin=0.8in,rmargin=0.8in]{geometry}
\usepackage{aliascnt}
\usepackage{hyperref}
\usepackage{verbatim}

\allowdisplaybreaks[4]

\newcommand{\vk}{\varkappa}

\newcommand{\BR}{\mathbb{R}}

\newcommand{\SL}{\sum\limits}

\newcommand{\al}{\alpha}
\newcommand{\be}{\beta}
\newcommand{\ga}{\gamma}
\newcommand{\de}{\delta}

\newcommand{\ME}{\mathbf E}

\newcommand{\CF}{\mathcal F}

\newcommand{\CG}{\mathcal G}
\newcommand{\MP}{\mathbf P}

\newcommand{\CA}{\mathcal A}

\newcommand{\CN}{\mathcal N}

\newcommand{\Oa}{\Omega}
\newcommand{\oa}{\omega}
\newcommand{\si}{\sigma}

\renewcommand{\phi}{\varphi}

\newcommand{\eps}{\varepsilon}

\newcommand{\Ra}{\Rightarrow}

\newcommand{\ol}{\overline}

\renewcommand{\comment}[1]{}

\newcommand{\md}{\mathrm{d}}

\begin{document}

\theoremstyle{plain}
\newtheorem{thm}{Theorem}[section]
\newtheorem*{thmnonumber}{Theorem}
\newtheorem{lemma}[thm]{Lemma}
\newtheorem{prop}[thm]{Proposition}
\newtheorem{cor}[thm]{Corollary}
\newtheorem{open}[thm]{Open Problem}

\theoremstyle{definition}
\newtheorem{defn}{Definition}
\newtheorem{asmp}{Assumption}
\newtheorem{notn}{Notation}
\newtheorem{prb}{Problem}

\theoremstyle{remark}
\newtheorem{rmk}{Remark}
\newtheorem{exm}{Example}
\newtheorem{clm}{Claim}

\author{Cameron Bruggeman and Andrey Sarantsev}

\title[Penalty Method]{Penalty Method for Reflected Diffusions\\ on The Half-Line} 

\address{Department of Mathematics, Columbia University}

\email{cam.bruggeman@gmail.com}

\address{Department of Statistics and Applied Probability, University of California, Santa Barbara}

\email{sarantsev@pstat.ucsb.edu}

\keywords{Stochastic differential equation, reflected diffusion, reflected Brownian motion, weak convergence, scale function, penalty method}

\subjclass[2010]{Primary 60J60, secondary 60J55, 60J65, 60H10}

\date{October 12, 2016. Version 16}

\begin{abstract}
Consider a reflected diffusion on the positive half-line. We approximate it by solutions of stochastic differential equations using the penalty method: We emulate the ``hard barrier'' of reflection by a ``soft barrier'' of a large drift coefficient, which compells the diffusion to return to the positive half-line. The main tool of the proof is convergence of scale functions. 
\end{abstract}

\maketitle

\section{Introduction} 

Let us informally introduce the concept of a reflected diffusion process on the positive half-line $\BR_+ := [0, \infty)$. Take measurable functions $g : \BR_+ \to \BR$ and $\si : \BR_+ \to \BR_+$. An $\BR_+$-valued stochastic process $Z = (Z(t), t \ge 0)$ is called a {\it reflected diffusion on $\BR_+$} with {\it drift coefficient} $g$ and {\it diffusion coefficient} $\si$ if:

\medskip

(i) as long as $Z(t) > 0$, this process behaves as a solution to the following SDE:
$$
\md Z(t) = g(Z(t))\md t + \si(Z(t))\md W(t),
$$
where $W = (W(t), t \ge 0)$ is a Brownian motion on the real line;

\medskip

(ii) at $z = 0$, it is reflected in the positive direction.

\medskip

The precise definition is given in Section 2. By construction, this process $Z$ cannot assume negative values. In other words, it cannot penetrate ``the barrier at zero''. This article is devoted to the following question: can we ``approximate'' this ``hard barrier'' by a ``soft barrier'' created by a large drift coefficient? More specifically, consider the solution to the following SDE:
$$
\md X(t) = f(X(t))\md t + \tilde{\si}(X(t))\md W(t),\ \ X(0) = Z(0).
$$
Assume that $f(z) \approx g(z)$ for positive $z$ away from zero, $f(z)$ is positive and very large when $z \approx 0$, and $\tilde{\si}(z) \approx \si(z)$ for all $z \ge 0$. Then we can hope that $X \approx Z$ in distribution. A precise statement of this result is given in Theorem~\ref{thm1}. 

Let us give a simple example. Suppose we wish to approximate reflected Brownian motion $Z = (Z(t), t \ge 0)$ on the positive half-line by a solution of SDE. Assume $Z(0) = z > 0$. Consider the following sequence of diffusions:
\begin{equation}
\label{eq:simple-example}
\md X_n(t) = a_n1_{[-c_n, 0]}(X_n(t))\md t + \md W(t),\ \ X_n(0) = z.
\end{equation}
where $a_n, c_n > 0$ are such that 
\begin{equation}
\label{eq:simple-conditions}
a_n \to \infty,\ a_nc_n \to \infty.
\end{equation}
Then for every $T > 0$, the law of $X_n$ on the space $C[0, T]$ of continuous real-valued functions on $[0, T]$ weakly converges to the law of $Z$ on the same space. 

\medskip

In the literature, this is sometimes called the {\it penalty method}, with the {\it penalty function} $f(x) - g(x)$. For example, in the setting~\eqref{eq:simple-example}, the penalty function is $a_n1_{[-c_n, 0]}(x)$. This method aims to emulate the reflection by a strong drift which is directed inside the domain of reflection. Our companion paper \cite{MyOwn8} deals with the case of multidimensional reflected Brownian motion in a domain $D \subseteq \BR^d$, including the case of {\it oblique} (not normal) reflection.

The main idea of the proof is to consider {\it scale functions}:
$$
s_n(x) = \int_{c_n}^x\exp\left(-2\int_{b_n}^y\frac{f_n(z)}{\si_n^2(z)}\md z\right)\md y.
$$
The integration limits $c_n$ and $b_n$ of integration can be chosen arbitrarily, because the function $s_n$ is defined up to an additive and a multiplicative constant. With a certain choice of $s_n$, we have: 
$$
s_n(x) \to 
\begin{cases}
s(x),\ x > 0;\\
-\infty,\ x < 0,
\end{cases}
\ \ \mbox{where}\ \ 
s(x) := 
\int_0^x\exp\left(-2\int_0^y\frac{g(z)}{\si^2(z)}\md z\right)\md y
$$
can be viewed as the scale function for the reflected diffusion $Z$. 

\subsection{Historical review} Reflected diffusions in one or many dimensions were studied extensively since the 1960s. 
The concept of reflected diffusion was pioneered by Skorohod in the papers \cite{Skorohod1961a, Skorohod1961b}; see also the book \cite{SkorohodBook}, and an article \cite{McKean1963}. Without attempting to conduct an extensive survey, let us mention the following articles on multidimensional reflected Brownian motion in general regions: \cite{HR1981a, HW1987b, HW1987a, LS1984, SV1971, Tanaka1979, Watanabe1971a, Watanabe1971b,   Wil1995}. 

Now, let us survey the literature on the penalty method. The papers \cite{PardouxWilliams, WilliamsZheng1990} apply the theory of Dirichlet forms to use penalty method for the case of stationary processes. A related paper \cite{SoftlyRBM} also deals with stationary distributions for the penalized Brownian motion in a convex polyhedron, which is intended to approximate a semimartingale reflected Brownian motion in this polyhedron; the authors call this {\it soft reflection}. 
The paper \cite{Slo1} applies the penalty method to a multidimensional reflected diffusion in a convex domain $D$, with the penalty function $x \mapsto n(x - \Pi(x))$; here, $\Pi(x)$ is the projection of $x$ on $D$. In the companion paper \cite{MyOwn8} mentioned above, we generalize this approach to more general penalty functions. The papers \cite{Slo2, Slo3, Men1, Men2} apply a similar technique to stochastic differential equations with jumps; see also a related paper \cite{Slo4}. 
The penalty method with penalty function $x \mapsto nx_-$ was also used in \cite{Karoui} to approximate reflected backward stochastic differential equations (BSDE) on the real line with non-reflected BSDE and to prove existence of a solution of a reflected BSDE. In the mutlidimensional setting, similar work was carried out in \cite{Slo5, Slo6}. 
Finally, let us mention the articles \cite{BassHsu, Cepa,  Fukushima,  LS1984, Aux1, Shalaumov} on the penalty method. 

\subsection{Relation with skew Brownian motion}

A survey \cite{SkewSurvey} contains a few ways to construct a {\it skew Brownian motion}. This can be loosely described as a Brownian motion such that each excursion is independently filpped to the positive half-line with a certain probability $\al \in (0, 1)$, and to the negative half-line with probability $1 - \al$. Here, $\al$ is called the {\it skewness parameter}. This process is, in some sense, between a usual Brownian motion and a reflected Brownian motion: for $\al = 1/2$ we get the usual Brownian motion, and for $\al = 1$ we get a reflected Brownian motion. 

The survey \cite{SkewSurvey} contains some results about approximation of a skew Brownian motion by solutions of SDEs. In the setting of~\eqref{eq:simple-example}, if $a_nc_n \to k > 0$ as $n \to \infty$ instead of $a_nc_n \to \infty$,  then $X_n$ converges in law to a skew Brownian motion with the skewness parameter 
$$
\al(k) := \frac{e^{2k}}{1 + e^{2k}}.
$$
For $k \to \infty$, we have: $\al(k) \to 1$. Therefore, one expects the skew Brownian motion to converge in law to the reflected Brownian motion. This is a natural way to arrive at the conditions~\eqref{eq:simple-conditions}. Related papers \cite{Portenko1979, Portenko1979a} deal with multidimensional skew Brownian motions.

\subsection{Organization of the paper} Section 2 is devoted to the setup of the problem and the main result (Theorem~\ref{thm1}), as well as corollaries and examples (including the one mentioned above). The proof of Theorem~\ref{thm1} is given in Section 3. The Appendix contains some technical lemmata.

\section{Definitions and the Main Result}

\subsection{Notation} Denote weak convergence by the arrow $\Ra$. Let $C[0, T]$ be the space of continuous functions $[0, T] \to \BR$. We let $\inf\varnothing = \infty$. A {\it standard Brownian motion} is a one-dimensional Brownian motion, starting from zero, with zero drift coefficient and unit diffusion coefficient. For an event $A$, let $A^c$ be its complement. Take a continuous function $f : [a, b] \to \BR$. For $\de > 0$, its {\it modulus of continuity} corresponding to $\de$ is denoted as follows:
$$
\oa(f, [a, b], \de) := \max\limits_{\substack{t_1, t_2 \in [a, b]\\ |t_1 - t_2| \le \de}}|f(t_1) - f(t_2)|.
$$

\subsection{The concept of a reflected diffusion}

We operate in a standard setting: a filtered probability space $(\Oa, \CF, (\CF_t)_{t \ge 0}, \MP)$, with the filtration $(\CF_t)_{t \ge 0}$ satisfying the usual conditions. Take a standard $(\CF_t)_{t \ge 0}$-Brownian motion $W = (W(t), t \ge 0)$. Consider two measurable functions $g : \BR_+ \to \BR$ and $\si : \BR_+ \to \BR_+$.   

\begin{defn} Take a continuous adapted process $Z = (Z(t), t \ge 0)$ with values in $\BR_+$, as well as another continuous adapted process $L = (L(t), t \ge 0)$, with the following properties:

\medskip

(i) $L(0) = 0$, $L$ is nondecreasing and can increase only when $Z = 0$; 
we can write the last property formally as 
$$
\int_0^{\infty}Z(t)\,\md L(t) = 0;
$$

\medskip

(ii) for all $t \ge 0$, 
$$
Z(t) = Z(0) + \int_0^tg(Z(s))\,\md s + \int_0^t\si(Z(s))\,\md W(s) + L(t).
$$

\medskip

Then the process $Z$ is called a {\it reflected diffusion} on $\BR_+$ with {\it drift coefficient $g$} and  {\it diffusion coefficient} $\si$. The process $L$ is called the {\it reflection term} corresponding to $Z$. If $Z(0) = z_0$, we say $Z$ {\it starts from} $z_0$.
\label{def:reflected-diffusion}
\end{defn}

\begin{rmk}
The differential $\md L$ can be viewed as ``the push'' which does not allow $Z$ to go negative when $Z$ is at zero.
\end{rmk}

There are many conditions on the coefficients $g$ and $\si^2$ which ensure the weak or strong existence and uniqueness of a reflected diffusion together with the corresponding reflection term, \cite{Skorohod1961a, Skorohod1961b, McKean1963}. We will just assume that weak existence and uniqueness in law hold. For the rest of the article, we fix $z_0 > 0$. 

\begin{asmp}
For this fixed $z_0 > 0$, the process $Z$ described in Definition~\ref{def:reflected-diffusion} exists in the weak sense and is unique in law.
\end{asmp}

To state the main result, we need some additional conditions on $g$ and $\si$.

\begin{asmp}
The functions $g$ and $\si$ are continuous on $\BR_+$, and $\si(x) > 0$ for all $x \ge 0$. 
\end{asmp}

\subsection{Main result}

Let again $W = (W(t), t \ge 0)$ be a standard $(\CF_t)_{t \ge 0}$-Brownian motion. Take two sequences $(f_n)_{n \ge 1}$ and $(\si_n)_{n \ge 1}$ of measurable functions
$$
f_n : \BR \to \BR,\ \ \si_n : \BR \to \BR_+,\ \ n = 1, 2, \ldots
$$
Consider a sequence $(X_n)_{n \ge 1}$ of approximating diffusions:
\begin{equation}
\label{eq:approx-diffusions}
\md X_n(t) = f_n(X_n(t))\,\md t + \si_n(X_n(t))\,\md W(t),\ X_n(0) = z_n.
\end{equation}

\begin{asmp}
For each $n = 1, 2, \ldots$, the coefficients $f_n, \si_n$ and the initial condition $z_n$ are such that the process $X_n$ defined in ~\eqref{eq:approx-diffusions} exists in the weak sense on the infinite time horizon and is unique in law. 
\end{asmp}

Let us also make an assumption about the initial conditions. We decided to state this separately from the conditions in the body of Theorem~\ref{thm1}, which deal with $f_n$ and $\si_n$. 

\begin{asmp} $z_n \to z_0$ as $n \to \infty$.
\end{asmp}

Now, let us state the main result of the paper.

\begin{thm}
\label{thm1}
Under Assumptions 1 - 4, suppose that:

\medskip

(i) there exists $\eps_0 > 0$ such that for every $\eps \in (0, \eps_0)$,
$$
\lim\limits_{n \to \infty}\int_{-\eps}^{\eps}f_n(x)\,\md x = \infty;
$$

\medskip

(ii) for every $[x_1, x_2] \subseteq (0, \infty)$, 
$$
\int_{x_1}^{x_2}|f_n(x) - g(x)|\,\md x \to 0,\ \ n \to \infty;
$$

\medskip

(iii) for every $x > 0$, we have: $\si_n \to \si$ uniformly on $[0, x]$;

\medskip

(iv) the family of functions $(\si_n)_{n \ge 1}$ is equicontinuous at $x = 0$, that is, 
$$
\lim\limits_{\de \to 0}\sup\limits_{n \ge 1}\sup\limits_{-\de \le x \le \de}|\si_n(x) - \si_n(0)| = 0.
$$

\medskip

(v) for every $\eps > 0$, there exist $n_0$ and $\de > 0$ such that
$$
f_n(x) > g(x) -\eps\ \ \mbox{for}\ \ n > n_0,\ |x| < \de.
$$

\medskip

Then for every $T > 0$, as $n \to \infty$,
$$
X_n \Ra Z\ \ \mbox{weakly in}\ \ C[0, T].
$$
\end{thm}

\subsection{The case of reflected Brownian motion}

Consider a special case, when $Z$ is a reflected Brownian motion on $\BR_+$, that is, $g \equiv 0$ and $\si \equiv 1$. Suppose we decide to approximate it by diffusion processes with diffusion coefficients $1$, that is, $\si_n \equiv 1,\ n = 1, 2, \ldots$. Then the statement of the theorem can be simplified a bit. 

\begin{cor} Under Assumption 3, suppose that:

\medskip

(i) there exists $\eps_0 > 0$ such that for every $\eps \in (0, \eps_0)$, we have:
$$
\lim\limits_{n \to \infty}\int_{-\eps}^{\eps}f_n(x)\,\md x = \infty;
$$

\medskip

(ii) for every $[x_1, x_2] \subseteq (0, \infty)$, we have:
$$
\lim\limits_{n \to \infty}\int_{x_1}^{x_2}|f_n(x)|\,\md x = 0;
$$

\medskip

(iii) for every $\eps > 0$, there exist $n_0$ and $\de > 0$ such that
$$
f_n(x) > -\eps\ \ \mbox{for}\ \ n > n_0,\ |x| < \de.
$$

\medskip

Then the diffusion processes $X_n$ defined in~\eqref{eq:approx-diffusions} converge weakly to a reflected Brownian motion on $\BR_+$ in $C[0, T]$, for every $T > 0$. 

\label{cor:conv-to-RBM}

\end{cor}

\subsection{Examples}

Now, let us present a few applications of Corollary~\ref{cor:conv-to-RBM}, starting with the example in the Introduction.

\begin{exm} Take two sequences $(a_n)_{n \ge 1}$ and $(c_n)_{n \ge 1}$ of positive numbers, and let
$$
f_n(x) = a_n1_{(-c_n, 0)}(x).
$$
Then the sequence $(f_n)_{n \ge 1}$ always satisfies conditions (ii) and (iii) of Corollary~\ref{cor:conv-to-RBM}. It is an easy exercise to see that condition (i) is satisfied if and only if 
\begin{equation}
\label{simpleexm}
\lim\limits_{n \to \infty}a_n = \infty,\ \lim\limits_{n \to \infty}a_nc_n = \infty.
\end{equation}
%Fix $\eps > 0$. Then 
%$$
%\int_{-\eps}^{\eps}f_n(x)\md x = a_n\left(c_n\wedge\eps\right) = a_nc_n\wedge\eps a_n.
%$$
%Therefore,
%$$
%\int_{-\eps}^{\eps}f_n(x)\md x \to \infty,\ \ \mbox{if and only if}\ \ 
%\lim\limits_{n \to \infty}\eps a_n = \infty\ \ \mbox{and}\ \ \lim\limits_{n \to \infty}a_nc_n = \infty,
%$$
%which is equivalent to~\eqref{simpleexm}. 
\label{exm1}
\end{exm}

\begin{exm} Let us ``erect the wall'' as in previous example, but to the right of $x = 0$: take two sequences $(a_n)_{n \ge 1}$ and $(c_n)_{n \ge 1}$ of positive numbers, and let
$$
f_n(x) = a_n1_{(0, c_n)}(x).
$$
Once again, this satisfies the conditions of Corollary~\ref{cor:conv-to-RBM} if and only if: 
%Then we can similarly prove that condition (i) from Corollary~\ref{cor:conv-to-RBM} is equivalent to~\eqref{simpleexm}. Condition (ii) is always satisfied, just as in the previous example. However, condition (iii) is not automatically satisfied; it is equivalent to 
%$$
%\lim\limits_{n \to \infty}c_n = 0.
%$$
%Indeed, if $c_n \to 0$ as $n \to \infty$, then $f_n \to 0$ uniformly on every $[x_1, x_2] \subseteq (0, \infty)$, so condition (iii) holds. However, if $c_n$ does not converge to $0$, then it is easy to see that condition (iii) from Corollary~\ref{cor:conv-to-RBM} does not hold, because $a_n \to \infty$ as $n \to \infty$. To summarize, we need to demand
\begin{equation}
\label{secondexm}
\lim\limits_{n \to \infty}a_n = \infty,\ \ \lim\limits_{n \to \infty}c_n = 0,\ \ \mbox{and}\ \ \lim\limits_{n \to \infty}a_nc_n = \infty.
\end{equation}
\label{exm2}
\end{exm}

\begin{exm} Take a measurable function $\psi : \BR \to \BR$ with the following properties:
\begin{equation}
\label{eq:properties-of-psi}
\varliminf\limits_{x \to 0}\psi(x) \ge 0;\ \ \lim\limits_{y \to \infty}\int_{-y}^y\psi(x)\md x =: I \in (0, \infty);\ \ \lim\limits_{x \to \infty}\psi(x) = 0.
\end{equation}
Take two sequences $(a_n)_{n \ge 1}$ and $(c_n)_{n \ge 1}$ of positive real numbers which satisfy~\eqref{secondexm} and, in addition,
\begin{equation}
\label{eq:thirdexm}
\lim\limits_{n \to \infty}\left[a_n\sup\limits_{x \ge x_0/c_n}|\psi(x)|\right] = 0\ \ \mbox{for every}\ \ x_0 > 0.
\end{equation}
Define the following sequence of functions:
$$
f_n(x) := a_n\psi(x/c_n),\ \ n = 1, 2, \ldots
$$
Then this sequence of functions satisfies the conditions of Corollary~\ref{cor:conv-to-RBM}. Indeed, let us verify condition (i): for every $\eps > 0$, as $n \to \infty$,
$$
\int_{-\eps}^{\eps}f_n(x)\md x = a_nc_n\int_{-\eps/c_n}^{\eps/c_n}\psi(y)\md y \to \infty,
$$
because from~\eqref{secondexm} and~\eqref{eq:properties-of-psi} we get:
$$
a_nc_n \to \infty\ \ \mbox{and}\ \ \int_{-\eps/c_n}^{\eps/c_n}\psi(y)\md y \to I > 0.
$$
Next, condition (ii) follows from~\eqref{eq:thirdexm}, and condition (iii) follows from $\varliminf_{x \to 0}\psi(x) \ge 0$. Examples 1 and 2 are actually two particular cases of this more general example, with 
\begin{equation}
\label{eq:exm-of-psi}
\psi(x) = 1_{[-1, 0]}(x)\ \ \mbox{and}\ \ \psi(x) = 1_{[0, 1]}(x),
\end{equation}
correspondingly. Note that~\eqref{eq:thirdexm} follows automatically from~\eqref{secondexm} when $\psi$ has compact support, as in~\eqref{eq:exm-of-psi}. There are other examples of functions $\psi$ which satisfy~\eqref{eq:properties-of-psi}, such as $\psi(x) = e^{-|x|}1_{(-\infty, 0]}$. This latter function was used to construct the so-called {\it exponential reflected Brownian motion} in the paper \cite{SoftlyRBM}. We consider this topic in our companion paper \cite{MyOwn8}.  
\label{exm3}
\end{exm}

\section{Proof of Theorem~\ref{thm1}}

\subsection{Outline of the proof} For the rest of this section, we fix the time horizon $T > 0$. Let us split the proof of the main result, Theorem~\ref{thm1}, into a few lemmata. Every subsection of this section will contain a proof of the corresponding lemma. In the current subsection, we enunciate these lemmata and show how they fit together. 

Extend the continuous functions $g, \si : \BR_+ \to \BR$ to the whole real line by making them constant on the negative half-line: $g(x) := g(0)$ and $\si(x) := \si(0)$ for $x \le 0$. First, we {\it localize}: namely, we stop processes $X_n$ when they hit a certain level $C > 0$. More formally, for every $C > 0$, let 
$$
T_C^{(n)} := \inf\{t \in [0, T]\mid X_n(t) = C\}\wedge T,\ \ n = 1, 2, \ldots
$$
Denote the stopped processes by
$$
X_n^C(t) \equiv X_n(t\wedge T_C^{(n)}).
$$

\begin{lemma} Suppose we proved that for every $C> 0$, every weak limit point of $(X_n^C)_{n \ge 1}$ in $C[0, T]$ behaves as the reflected diffusion $Z$ until it hits $C$. Then $X_n \Ra Z$ in $C[0, T]$. 
\label{lemma:localization}
\end{lemma}

\begin{rmk} Note that we are not concerned with behavior of this weak limit point {\it after} it hits $C$. It might be stopped there, or it might still move after this hitting moment. We only need this weak limit point to behave as the reflected diffusion $Z$ {\it before} it hits $C$. 
\end{rmk}

Now, fix $C > 0$. The next part of the proof is devoted to showing that  for large $n$, the process $X^C_n$ does not get far away into the negative half-line $(-\infty, 0)$. 

\begin{lemma}
\label{lemma:not-far-into-negative}
For every $\eps > 0$, 
$$
\lim\limits_{n \to \infty}\MP\left(\min\limits_{0 \le t \le T}X^C_n(t) \le -\eps\right) = 0.
$$
\end{lemma}

\medskip

For $n = 1, 2, \ldots$ and $t \in [0, T]$, denote
$$
L_n(t) = \int_0^{t\wedge T^{(n)}_C}\left[f_n(X_n(u)) - g(X_n(u))\right]\,\md u,\ \ 
$$
\begin{equation}
\label{eq:Zn}
Z_n(t) = X^C_n(t) - L_n(t) = z_n + \int_0^{t\wedge T^{(n)}_C}g(X_n^C(u))\,\md u + \int_0^{t\wedge T^{(n)}_C}\si_n(X_n(u))\,\md W(u).
\end{equation}

\begin{lemma}
\label{lemma:Xn-tight}
The sequence $(X_n^C)_{n \ge 1}$ is tight in $C[0, T]$.
\end{lemma}

Fix an increasing sequence $(n_k)_{k \ge 1}$ of positive integers. 

\begin{lemma} 
\label{lemma:Zbar-conv} 
We can find a subsequence $(n'_k)_{k \ge 1}$ of $(n_k)_{k \ge 1}$ such that, after changing the probability space, we have the following a.s. convergence in $C[0, T]\times C[0, T]$:
$$
\left(X^C_{n'_k}, Z_{n'_k}\right) \to (\ol{Z}, \tilde{Z}),\ \ k \to \infty.
$$
The limiting processes $\ol{Z}$ and $\tilde{Z}$ satisfy the following relation: for
$$
t < T^{\ol{Z}}_C := \inf\{t \in [0, T]\mid \ol{Z}(t) = C\}\wedge T,
$$
we have:
$$
\tilde{Z}(t) = z_0 + \int_0^{t}g(\ol{Z}(u))\,\md u + \int_0^{t}\si(\ol{Z}(u))\,\md \ol{W}(u),
$$
where $\ol{W} = (\ol{W}(u), u \ge 0)$ is a certain standard Brownian motion. 
\end{lemma}

Now, a.s. uniformly on $[0, T]$, as $k \to \infty$, 
$$
L_{n'_k} \equiv X^C_{n'_k} - Z_{n'_k} \to L := \ol{Z} - \tilde{Z}.
$$
Since $L_{n'_k}(0) = 0$ for all $k = 1, 2, \ldots$, we have: $L(0) = 0$. 

\begin{lemma} On the interval $[0, T^{\ol{Z}}_C]$, the process $\ol{L}$ is a.s. nondecreasing, and it can increase only when $\ol{Z} = 0$. 
\label{lemma:boundary-property-of-L}
\end{lemma}

Combining the results of Lemmata~\ref{lemma:Zbar-conv}, and~\ref{lemma:boundary-property-of-L}, we get the following result: The process $\ol{Z}$ is a version of the reflected diffusion on $\BR_+$ with drift coefficient $g$ and diffusion coefficient $\si$, starting from $z$, at least until it hits the level $C$. We arrive at the conclusion that for every subsequence $(n_k)_{k \ge 1}$ there exists a subsequence $(n'_k)_{k \ge 1}$ such that $X_{n'_k}^C$ weakly converges (in $C[0, T]$, as $k \to \infty$) to a process $\ol{Z}$ is a version of this reflected diffusion $Z$, at least until $\ol{Z}$ hits $C$. Use Lemma~\ref{lemma:localization} to complete the proof of Theorem~\ref{thm1}. 

\subsection{Proof of Lemma~\ref{lemma:localization}}

Take a small probability $\eta > 0$. Then there exists $C > 0$ such that 
$$
\MP\left(\max\limits_{0 \le t \le T}Z(t) < C\right) \ge 1 - \eta.
$$
Fix an increasing sequence $(n_k)_{k \ge 1}$ of positive integers. Then it has a subsequence $(n'_k)_{k \ge 1}$ such that , in $C[0, T]$,
$$
X^C_{n'_k} \Ra \ol{Z}, \ \ k \to \infty.
$$
By assumption of this lemma, the process $\ol{Z}$ behaves as $Z$ until it hits the level $C$. Therefore,
$$
\MP\left(\max\limits_{0 \le t \le T}\ol{Z}(t) < C\right) = \MP\left(\max\limits_{0 \le t \le T}Z(t) < C\right) \ge 1 - \eta.
$$
The set $\CA := \{h \in C[0, T]\mid \max_{t \in [0, T]}h(t) < C\}$ is open in $C[0, T]$. For every measurable subset $\mathcal B \subseteq \CA$, we have:
\begin{equation}
\label{eq:equality-for-Z}
\MP(\ol{Z} \in \mathcal B) = \MP(Z \in \mathcal B).
\end{equation}
Note that we can write
$$
\{Z \in \CA\} = \left\{\max\limits_{0 \le t \le T}Z(t) < C\right\} = \left\{\max\limits_{0 \le t \le T}Z^C(t) < C\right\} = \{Z^C \in \CA\}.
$$
Therefore, $\MP(\ol{Z} \in \CA) = \MP(Z \in \CA) > 1 - \eta$. 
%By the portmanteau theorem, see \cite[Section 1.2]{BillingsleyBook},
%$$
%\varliminf\limits_{k \to \infty}\MP\left(X^C_{n_k} \in A\right) \ge \MP(\ol{Z} \in \CA) > 1 - \eta.
%$$
%So there exists $k_{\eta}$ such that for $k > k_{\eta}$, 
%$$
%\MP(\max\limits_{0 \le t \le T}X_{n_k}(t) < C) \equiv \MP(X_{n_k}^C \in \CA) > 1 - \eta.
%$$
For any open set $\CG \subseteq C[0, T]$, we have:
$$
\MP(X_n \in \CG)  \ge \MP(X_n \in \CA\cap \CG) = \MP(X_n^C \in \CA\cap \CG).
$$
The set $\CA\cap \CG$ is also open in $C[0, T]$. Applying the portmanteau theorem, see \cite[Section 1.2]{BillingsleyBook},
\begin{equation}
\label{eq:1}
\varliminf\limits_{k \to \infty}\MP(X_{n'_k} \in \CG)  \ge \varliminf\limits_{k \to \infty}\MP(X^C_{n'_k} \in \CA\cap\CG) \ge \MP(\ol{Z} \in \CA\cap\CG).
\end{equation}
Applying~\eqref{eq:equality-for-Z} to $\mathcal B := \CA\cap\CG$, we get:
\begin{equation}
\label{eq:2}
\MP(\ol{Z} \in \CA\cap\CG) = \MP(Z \in \CA\cap\CG).
\end{equation}
Finally, it is straightforward to get the following estimate:
\begin{equation}
\label{eq:3}
\MP(Z \in \CA\cap\CG) \ge \MP(Z \in \CG) - \MP(Z \notin \CA) \ge \MP(Z \in \CG) - \eta.
\end{equation}
Combining~\eqref{eq:1},~\eqref{eq:2},~\eqref{eq:3}, we get:  
$$
\varliminf\limits_{k \to \infty}\MP(X_{n'_k} \in \CG) \ge \MP(Z \in \CG) - \eta.
$$
Now, a small difficulty arises: the subsequence $(n'_k)_{k \ge 1}$ might depend on $\eta$. But we can bypass this by taking a sequence $\eta_m := 1/m$ and constructing a sequence of inserted subsequences $(n_k^{(m)})_{k \ge 1}$ such that $(n_k^{(m+1)})_{k \ge 1}$ is a subsequence of $(n^{(m)}_k)_{k \ge 1}$. Then 
$$
\varliminf\limits_{k \to \infty}\MP(X_{n^{(m)}_k} \in \CG) \ge \MP(Z \in \CG) - \frac1m.
$$
Take the diagonal subsequence $\ol{n}_k := n^{(k)}_k,\ k = 1, 2, \ldots$ Then 
$$
\varliminf\limits_{k \to \infty}\MP(X_{\ol{n}_k} \in \CG) \ge \MP(Z \in \CG).
$$
This is true for every open set $\CG \subseteq C[0, T]$. By another application of the portmanteau theorem, $X_{\ol{n}_k} \Ra Z$. We have shown that every increasing sequence $(n_k)_{k \ge 1}$ of positive integers contains a subsequence $(\ol{n}_k)_{k \ge 1}$ such that $X_{\ol{n}_k} \Ra Z$. This proves that $X_n \Ra Z$. 

\subsection{Proof of Lemma~\ref{lemma:not-far-into-negative}} We split this subsection into a few parts. In the first part, we give an overview of the proof: we formulate a series of auxillary lemmata and explain how they fit together to make a proof of Lemma~\ref{lemma:not-far-into-negative}. In the subsequent parts, we prove each of these auxillary lemmata. 

\subsubsection{Overview of the proof of Lemma~\ref{lemma:not-far-into-negative}} 

The key tool of the proof is convergence of scale functions. The scale function $s_n$ for the diffusion $X_n$ is defined as follows:
$$
s_n(x) = \int_{c_n}^x\exp\left(-2\int_{b_n}^y\frac{f_n(z)}{\si_n^2(z)}\,\md z\right)\md y.
$$
Here, the lower limits $c_n$ and $b_n$ of integration can be chosen arbitrarily, because the function $s_n$ is defined up to an additive and a multiplicative constant. For our purposes, we choose $c_n = 1$ and $b_n > 0$ to be determined from the following lemma.

\begin{lemma}
\label{lemma:choice-of-bn}
There exists a sequence $(b_n)_{n \ge 1}$ such that $b_n > 0$ and $b_n \to 0$ as $n \to \infty$, and for every $x > 0$ we have:
$$
\lim\limits_{n \to \infty}\int_{b_n}^{x}\left|\frac{f_n(z)}{\si^2(z)} - \frac{g(z)}{\si^2(z)}\right|\md z = 0.
$$
\end{lemma}

Define the limiting scale function: 
$$
s(x) = \int_0^x\exp\left(-2\int_0^y\frac{g(z)}{\si^2(z)}\,\md z\right)\md y.
$$
Recall that we extended the functions $g$ and $\si$ to the real line. Therefore, we can define $s(x)$ using this formula for all $x \in \BR$. With this choice of $b_n$ as in Lemma~\ref{lemma:choice-of-bn}, 
we can prove convergence of scale functions:

\begin{lemma}
\label{lemma:conv-of-scale-functions}
For every $[x_1, x_2] \subseteq (0, \infty)$, uniformly on $[x_1, x_2]$,
$$
\lim\limits_{n \to \infty}s_n = s,\ \ \lim\limits_{n \to \infty}s'_n = s'.
$$
In addition, for $x < 0$, $\lim\limits_{n \to \infty}s_n(x) = -\infty$.  
\end{lemma}

The next lemma shows that for large $n$, the process $X_n$ is more and more likely to hit any negative level only after hitting any fixed level $a > z_0$.  

\begin{lemma} 
For every $\eps > 0$ and $a > z_0$, we have:
$$
\lim\limits_{n \to \infty}\MP\left(T^{(n)}_{-\eps} > T^{(n)}_a\right) = 1.
$$
\label{lemma:conv-of-hitting-times}
\end{lemma}

But if we let $a = C$, then 
$$
\left\{T^{(n)}_{-\eps} > T^{(n)}_a\right\} = \left\{\min\limits_{t \ge 0}X^C_n(t) > -\eps\right\} \subseteq \left\{\min\limits_{0 \le t \le T}X^C_n(t) > -\eps\right\}.
$$
This completes the proof of Lemma~\ref{lemma:not-far-into-negative}. 

\begin{rmk}
\label{rmk:sigmas}
For every $x > 0$, $\si_n \to \si$ uniformly on $[0, x]$, and $\si$ is a positive continuous function on $[0, x]$. Therefore, there exist $n_x$, $\ol{\si}_x$ and $\underline{\si}_x$ such that for $n \ge n_x$ and $y \in [0, x]$, we have:
\begin{equation}
\label{eq:def-of-sigmas}
0 < \underline{\si}_x \le \si_n(y) \le \ol{\si}_x < \infty,\ \ 0 < \underline{\si}_x \le \si(y) \le \ol{\si}_x < \infty.
\end{equation}
\end{rmk}

\subsubsection{Proof of Lemma~\ref{lemma:choice-of-bn}}

Let us split the proof into the two following lemmata. First, we show convergence of integrals over any fixed subinterval. 

\begin{lemma}
\label{lemma:conv-on-fixed-intervals}
For every $[x_1, x_2] \subseteq (0, \infty)$, 
\begin{equation}
\label{eq:conv-with-constant-limits}
\lim\limits_{n \to \infty}\int_{x_1}^{x_2}\left|\frac{f_n(z)}{\si_n^2(z)} - \frac{g(z)}{\si^2(z)}\right|\md z = 0.
\end{equation}
\end{lemma}

Then we show the existence of the sequence $(b_n)_{n \ge 1}$ with the required properties. It suffices to prove the following auxillary lemma and apply the results to 
$$
g_n(z) := \frac{f_n(z)}{\si_n^2(z)} - \frac{g(z)}{\si^2(z)}.
$$

\begin{lemma}
\label{lemma:b-n}
Take a sequence $(g_n)_{n \ge 1}$ of measurable functions $g_n : (0, \infty) \to \BR$ such that for every $[x_1, x_2] \subseteq (0, \infty)$, we have:
$$
\lim\limits_{n \to \infty}\int_{x_1}^{x_2}|g_n(z)|\,\md z = 0.
$$
Then there exists a sequence $(b_n)_{n \ge 1}$ of positive numbers with $b_n \to 0$ and for all $x_0 > 0$, 
$$
\lim\limits_{n \to \infty}\int_{b_n}^{x_0}|g_n(z)|\,\md z = 0.
$$
\end{lemma}

{\it Proof of Lemma~\ref{lemma:conv-on-fixed-intervals}.} We have:
\begin{align*}
\int_{x_1}^{x_2}\left|\frac{f_n(z)}{\si_n^2(z)} - \frac{g(z)}{\si^2(z)}\right|\md z &\le \int_{x_1}^{x_2}\frac{|f_n(z) - g(z)|}{\si_n^2(z)}\,\md z \\ & + 
\int_{x_1}^{x_2}|g(z)|\left|\si_n^{-2}(z) - \si^{-2}(z)\right|\md z := I_1(n) + I_2(n).
\end{align*}
Recall from~\eqref{eq:def-of-sigmas} that $\si^2_n(y) \ge \underline{\si}_{x_2}^2$ for all $y \in [x_1, x_2]$ and $n \ge n_{x_2}$. Using the condition (ii) of Theorem~\ref{thm1}, we get: as $n \to \infty$,
$$
I_1(n) \le \underline{\si}^{-2}_{x_2}\int_{x_1}^{x_2}\left|f_n(z) - g(z)\right|\md z \to 0.
$$
Since $\si_n^2 \to \si^2$ uniformly on $[x_1, x_2]$, by Lemma~\ref{lemma:aux} (ii), $\si_n^{-2} \to \si^{-2}$ uniformly on $[x_1, x_2]$. And the function $g$ is continuous (therefore, it is bounded) on $[x_1, x_2]$. Thus, uniformly on $[x_1, x_2]$, 
$$
|g(z)|\left|\si_n^{-2}(z) - \si^{-2}(z)\right| \to 0.
$$
Which proves that $I_2(n) \to 0$ as $n \to \infty$. $\square$

\medskip

{\it Proof of Lemma~\ref{lemma:b-n}.} For $b \in (0, 1)$ and $n = 1, 2, \ldots$, let
$$
M(b, n) := \int_b^1|g_n(z)|\,\md z.
$$
Then for every $b \in (0, 1)$, we have: $\lim_{n \to \infty}M(b, n) = 0$. Take $b := 1/k$ and find $n_k$ such that for $n \ge n_k$ we have: $M(1/k, n_k) \le 1/k$. Now, let $\ol{n}_1 := n_1$ and $\ol{n}_{k+1} := \max(\ol{n}_1, \ldots, \ol{n}_k, n_{k+1}) + 1$. Then $\ol{n}_1 < \ol{n}_2 < \ldots$, and for $n \ge \ol{n}_k$, $M\left(1/k, \ol{n}_k\right) \le 1/k$. 
Define the sequence $(b_n)_{n \ge 1}$ as follows: $b_{n} = k^{-1}$ for  $\ol{n}_{k} \le  n < \ol{n}_{k+1}$. 
For $n < \ol{n}_1$, just let $b_n = 1$. Then we get: for $\ol{n}_{k} \le  n < \ol{n}_{k+1}$, 
$$
M(b_n, n) = M(1/k, n) \le \frac1k.
$$
Therefore, $M(b_n, n) \to 0$ as $n \to \infty$. This proves the statement of the lemma for $x_0 \le 1$. For $y > 1$, just note that 
$$
\int_1^{x_0}|g_n(z)|\,\md z \to 0.
$$
Thus, we can represent the original integral as
$$
\int_{b_n}^{x_0}|g_n(z)|\,\md z = \int_{b_n}^1|g_n(z)|\,\md z + \int_1^{x_0}|g_n(z)|\,\md z \to 0. \ \ \square
$$

\subsubsection{Proof of Lemma~\ref{lemma:conv-of-scale-functions}}

First, let us show that for $x > 0$, $\lim\limits_{n \to \infty}s'_n(x) = s'(x)$. Indeed, by Lemma~\ref{lemma:choice-of-bn}, we get: as $n \to \infty$, 
$$
\int_{b_n}^x\left|\frac{f_n(z)}{\si_n^2(z)} - \frac{g(z)}{\si^2(z)}\right|\md z \to 0.
$$
But we have:
\begin{equation}
\label{eq:comp-of-integrals-for-scale}
\left|\int_{b_n}^x\frac{f_n(z)}{\si_n^2(z)}\,\md z - \int_0^x\frac{g(z)}{\si^2(z)}\,\md z\right| \le \int_{b_n}^x\left|\frac{f_n(z)}{\si_n^2(z)} - \frac{g(z)}{\si^2(z)}\right|\md z + \int_0^{b_n}\frac{|g(z)|}{\si^2(z)}\,\md z.
\end{equation}
We need only to show that, as $n \to \infty$, 
\begin{equation}
\label{eq:integral-to-zero}
\int_0^{b_n}\frac{|g(z)|}{\si^2(z)}\,\md z \to 0.
\end{equation}
To prove~\eqref{eq:integral-to-zero}, we need only to note that $b_n \to 0$ as $n \to \infty$, $\si^{-2}(z) \le \underline{\si}^{-2}_1$ for $z \in [0, 1]$ because of~\eqref{eq:def-of-sigmas}, and the function $g$ is continuous. From~\eqref{eq:comp-of-integrals-for-scale}, it follows that, as $n \to \infty$, for all $x > 0$,
$$
\log s'_n(x) = -2\int_{b_n}^x\frac{f_n(z)}{\si_n^2(z)}\md z \to -2\int_0^x\frac{g(z)}{\si^2(z)}\md z = \log s'(x).
$$
Therefore, $s'_n(x) \to s'(x)$. Now, we show this convergence is uniform on every $[x_1, x_2] \subseteq (0, \infty)$. There exists an $n_0$ such that for $n \ge n_0$ we have: $b_n < x_1$. Then for $x \in [x_1, x_2]$, from~\eqref{eq:comp-of-integrals-for-scale}, we get:
$$
\left|\log s'_n(x) - \log s'(x)\right| \le \int_{b_n}^{x_2}\left|\frac{f_n(z)}{\si_n^2(z)} - \frac{g(z)}{\si^2(z)}\right|\md z + \int_0^{b_n}\frac{|g(z)|}{\si^2(z)}\md z \to 0,\ \ n \to \infty. 
$$
Therefore, $\log s'_n(x) \to \log s'(x)$ uniformly on $[x_1, x_2]$. The function $\log s'(x)$ is continuous on $[x_1, x_2]$, and therefore is bounded on this segment. There exist $n_2$ and $C_0 > 0$ such that for $x \in [x_1, x_2]$, $n \ge n_2$, we have:
$$
\left|\log s'_n(x)\right| \le C_0,\ \ \left|\log s'(x)\right| \le C_0.
$$
The function $z \mapsto e^z$ is uniformly continuous on $[-C_0, C_0]$, and $\log s'_n(x) \to \log s_n(x)$ uniformly on $[x_1, x_2]$. Therefore, uniformly on $[x_1, x_2]$,
$$
s'_n(x) = e^{\log s'_n(x)} \to e^{\log s'(x)} = s'(x),\ \ n \to \infty,\ \ \mbox{and}
$$
$$
s_n(x) = \int_1^xs'_n(z)\md z \to s(x) = \int_1^xs'(z)\md z,\ \ n \to \infty,
$$
and the convergence is uniform on every $[x_1, x_2] \subseteq (0, \infty)$. Finally, let us show that $s_n(x) \to -\infty$ for $x < 0$. Indeed, if we prove that $s'_n(x) \to \infty$ for $x < 0$, then $s'_n(x) \ge 0$, and by Fatou's lemma,  
$$
-s_n(x) = -\int_1^xs'_n(y)\md y = \int_x^1s'_n(y)\md y \ge \int_x^0 s'_n(y)\md y \to \infty.
$$
It suffices to show that 
\begin{equation}
\label{eq:4}
\log s'_n(x) = -2\int_{b_n}^x\frac{f_n(z)}{\si_n^2(z)}\md z = 2\int_x^{b_n}\frac{f_n(z)}{\si_n^2(z)}\md z \to \infty.
\end{equation}
But we have:
\begin{equation}
\label{eq:5}
\int_x^{b_n}\frac{f_n(z)}{\si_n^2(z)}\md z = \left(\int_x^1 - \int_{b_n}^1\right)\frac{f_n(z)}{\si_n^2(z)}\md z,\ \ \mbox{and}\ \ \int_x^1\frac{f_n(z)}{\si_n^2(z)}\md z \to \infty.
\end{equation}
From the relation~\eqref{eq:comp-of-integrals-for-scale} applied to $x = 1$, we get:
\begin{equation}
\label{eq:6}
\int_{b_n}^1\frac{f_n(z)}{\si_n^2(z)}\md z \to \int_0^1\frac{g(z)}{\si^2(z)}\md z.
\end{equation}
Combining~\eqref{eq:5} and~\eqref{eq:6}, we get~\eqref{eq:4}. 

\subsubsection{Proof of Lemma~\ref{lemma:conv-of-hitting-times}}

The process $s_n(X_n(\cdot))$ is a local martingale. Therefore, the process
$s_n(X_n(\cdot\wedge T^{(n)}_{-\eps}\wedge T^{(n)}_a))$ is a bounded martingale. By the optional stopping theorem,
\begin{equation}
\label{eq:fraction}
\MP\left(T^{(n)}_{-\eps} < T^{(n)}_a\right) = \frac{s_n(a) - s_n(z_0)}{s_n(a) - s_n(-\eps)}.
\end{equation}
But by Lemma~\ref{lemma:conv-of-scale-functions}, $s_n(a) \to s(a),\ s_n(z_0) \to s(z_0),\ s_n(-\eps) \to -\infty$. Therefore, the quantity in~\eqref{eq:fraction} converges to zero, which completes the proof. 

\subsection{Proof of Lemma~\ref{lemma:Xn-tight}}

Since $X^C_n(0) = z_n \to z_0 = Z(0)$, by the Arzela-Ascoli criterion it suffices to show that for all $\eps > 0$, 
$$
\lim\limits_{\de \to 0}\sup\limits_{n \ge 1}\MP\left(\oa(X^C_n, [0, T], \de) > 3\eps\right) = 0.
$$
Note that $X^C_n$ is constant on $[T^{(n)}_C, T]$, and is equal to $X_n$ on $[0, T^{(n)}_C]$. Therefore, $\oa(X^C_n, [0, T], \de) = \oa(X_n, [0, T^{(n)}_C], \de)$. 
Assume $\oa(X_n, [0, T^{(n)}_C], \de) > 3\eps$, and $\min_{[0, T]}X^C_n > -\eps$. Then there exist $t_1, t_2 \in [0, T^{(n)}_C]$ such  that $|t_1 - t_2| \le \de$ , and $|X_n(t_1) - X_n(t_2)| > 3\eps$. Assume without loss of generality $X_n(t_1) - X_n(t_2) > 3\eps$. Since $X_n(t_2) > -\eps$, we get: $X_n(t_1) > 2\eps$. By continuity of $X_n$, we can find $t_3$ between $t_1$ and $t_2$ such that $X_n(t_3) = X_n(t_2) + 2\eps > 3\eps - \eps = \eps$. And we can find $t_4$ between $t_1$ and $t_3$ such that $X_n(t_4) = X_n(t_3) + \eps$. Then 
$$
X_n(t_3), X_n(t_4) \in [\eps, C],\ \ |t_4 - t_3| \le |t_1 - t_2| \le \de.
$$
But $s'_n(x) > 0$ for $x \in [\eps, C]$ and $n = 1, 2, \ldots$, and $s'_n(x) \to s'(x) > 0$ uniformly on $[\eps, C]$. Therefore, there exists $C_1 > 0$ such that for all $n = 1, 2, \ldots$, $s'_n(x) \ge C_1,\ x \in [\eps, C]$. We have:
\begin{equation}
\label{eq:difference-between-scale-functions}
s_n(X_n(t_4)) - s_n(X_n(t_3)) \ge C_1(X_n(t_4) - X_n(t_3)) \ge C_1\eps.
\end{equation}
But the process $s_n(X_n(\cdot))$ satisfies the following stochastic equation:
\begin{equation}
\label{eq:SDE-for-s_n(X_n)}
\md s_n(X_n(t)) = s'_n(X_n(t))\si_n(X_n(t))\md W(t),
\end{equation}
with the coefficient satisfying
\begin{equation}
\label{eq:estimate-on-diffusion}
\left|s'_n(x)\si_n(x)\right| \le C_0,\ \ x \in [\eps, C],\ \ n = 1, 2, \ldots
\end{equation}
%By Lemma~\ref{lemma:main-aux}, we have:
%$$
%\MP\left(\exists t_3, t_4 \in [0, T]:\ |t_3 - t_4| \le \de,\ s_n(X_n(t_4)) - s_n(X_n(t_3)) = C_1\eps\right) \le \varkappa\frac{(C - \eps)C_0^5T^{3/2}}{(C_1\eps)^6}\de.
%$$
%Therefore, for all $n \ge 1$, 
%$$
%\MP\left(\oa(X^C_n, [0, T], \de) \ge 3\eps\right) \le \varkappa\frac{(C - \eps)C_0^5T^{3/2}}{(C_1\eps)^6}\de.
%$$
By Lemma~\ref{lemma:main-aux} (proved in the Appendix), we have: for a certain real constant $\vk > 0$, 
$$
\MP\left(\exists\, t_3, t_4 \in [0, T]:\ |t_3 - t_4| \le \de,\ s_n(X_n(t_4)) - s_n(X_n(t_3)) = C_1\eps\right) \le \varkappa\frac{C_0^4T}{(C_1\eps)^2}\de.
$$
In other words, because of the estimate~\eqref{eq:estimate-on-diffusion}, the probability of the event~\eqref{eq:difference-between-scale-functions} is very small for small $\de$. 
Therefore, for all $n \ge 1$, 
$$
\MP\left(\oa(X^C_n, [0, T], \de) \ge 3\eps\right) \le \varkappa\frac{C_0^4T}{(C_1\eps)^2}\de.
$$
Letting $\de \to 0$, we complete the proof of Lemma~\ref{lemma:Xn-tight}. 

\subsection{Proof of Lemma~\ref{lemma:Zbar-conv}}
For $t \in [0, T]$ and $n = 1, 2,\ldots$, let
$$
M_n(t) := \int_0^{t\wedge T^{(n)}_C}\si_n(X_n(u))\md W(u),\ \ M(t) := \int_0^{t\wedge T^{\ol{Z}}_C}\si(\ol{Z}(u))\md W(u).
$$
These are continuous local martingales. We can express them as time-changed Brownian motions:
$$
M_n(t) = B_n\left(\langle M_n\rangle_t\right),\ \ M(t) = B\left(\langle M\rangle_t\right),
$$
where $B_n, B$ are standard Brownian motions, and the quadratic variations $\langle M_n\rangle_t$ and $\langle M\rangle_t$ of $M_n$ and $M$ respectively can be expressed explicitly as
$$
\langle M_n\rangle_t = \int_0^{t\wedge T^{(n)}_C}\si_n^2(X_n(u))\md u,\ \ \langle M\rangle_t = \int_0^{t\wedge T^{\ol{Z}}_C}\si^2(\ol{Z}(u))\md u.
$$
Note that $\si^2(\ol{Z}(u)) \le \ol{\si}^2_C$ for $u \in [0, T^{\ol{Z}}_C]$ (see Remark ~\ref{rmk:sigmas}). Therefore, $\langle M\rangle_t \le \ol{\si}^2_CT$ for $t \in [0, T]$, $n = 1, 2, \ldots$ We can conclude that the sequence $(X_{n_k}, B_{n_k})_{k \ge 1}$ is tight in $C([0, T], \BR)\times C([0, \ol{\si}^2_CT + 1], \BR)$. It has a weak limit point for some subsequence $(n'_k)$ of $(n_k)$:
$$
\left(X_{n'_k}, B_{n'_k}\right) \Ra (\ol{Z}, \ol{B}).
$$
The process $\ol{B}$ is a weak limit of a sequence of standard Brownian motions, and therefore is itself a standard Brownian motion. By the Skorohod representation theorem, we can switch from weak to pathwise convergence (a.s.) after changing the probability space. Since $X_{n'_k} \to \ol{Z}$, then by Lemma~\ref{lemma:not-far-into-negative} we have: $\ol{Z}(t) \ge 0$ for all $t \in [0, T]$. 

\begin{lemma} The hitting times of the level $C$ satisfy the following inequality a.s.:
$$
T^{\ol{Z}}_C \le \varliminf\limits_{k \to \infty}T^{(n'_k)}_C.
$$
\label{lemma:conv-of-hitting-moments}
\end{lemma}

\begin{proof} Indeed, assume the converse. Then there exists a subsequence $(\ol{n}_k)_{k \ge 1}$ of $(n'_k)_{k \ge 1}$ and a number $t_0 \in [0, T^{\ol{Z}}_C)$ such that $T^{(\ol{n}_k)}_C \le t_0$. But this means that 
$T^{(\ol{n}_k)}_C \le t_0 < T$, and 
$$
X^C_{\ol{n}_k}(t_0) = X^C_{\ol{n}_k}\left(T^{(\ol{n}_k)}_C\right) = C.
$$
Letting $k \to \infty$, we get: $\ol{Z}(t_0) = C$. But this contradicts the fact that $T^{\ol{Z}}_C > t_0$. This completes the proof of this lemma.
\end{proof}

By Lemma~\ref{lemma:aux-sigma}, $\si_{n'_k}(X_{n'_k}(t)) \to \si(\ol{Z}(t))$ a.s. uniformly on $[0, T]$. Hence a.s. uniformly on $[0, T]$, 
$$
\langle M_{n'_k}\rangle_t = \int_0^t\si^2_{n'_k}(X_{n'_k}(u))\md u \to 
\langle M\rangle_t = \int_0^t\si^2(\ol{Z}(u))\md u.
$$
Applying Lemma~\ref{lemma:aux} (iv), we get:
\begin{equation}
\label{eq:convergence-of-martingales}
M_{n'_k}(t) = B_{n'_k}(\langle M_{n'_k}\rangle_t) \to M(t) = B(\langle M\rangle_t)
\end{equation}
a.s. uniformly on $[0, T]$. Also, $g$ is continuous on $\BR$, and $X_{n'_k} \to \ol{Z}$ uniformly on $[0, T]$, hence $g(X_{n'_k}(t)) \to g(\ol{Z}(t))$ uniformly on $[0, T]$. Therefore, 
\begin{equation}
\label{eq:convergence-of-drift-parts}
\int_0^tg(X_{n'_k}(u))\md u \to \int_0^tg(\ol{Z}(u))\md u,\ n \to \infty,
\end{equation}
uniformly on $[0, T]$. Finally, $z_n \to z$. Combining this with~\eqref{eq:convergence-of-martingales} and~\eqref{eq:convergence-of-drift-parts}, we get: $Z_{n'_k} \to Z$ uniformly on $[0, T]$. Finally, weak convergence in $C[0, T]$ follows from uniform convergence. 

%\subsection{Proof of Lemma~\ref{lemma:nondecrease}}

\subsection{Proof of Lemma~\ref{lemma:boundary-property-of-L}} This proof is a bit lengthy, so we organize it in the same way as the proof of Theorem~\ref{thm1}: In the first subsection, we split the proof into a few lemmata, and in the following subsections, we prove each of these lemmata. 

\subsubsection{Outline of the proof of Lemma~\ref{lemma:boundary-property-of-L}} Fix a subinterval $[t_1, t_2] \subseteq [0, T]$ and a number $\eta > 0$. 

\begin{lemma}
\label{lemma:prob-zero-for-C}
Suppose we have already shown that  
\begin{equation}
\label{eq:prob-zero-for-C}
\MP\bigl(|\ol{L}(t_2) - \ol{L}(t_1)| > \eta,\ \min\limits_{t \in [t_1, t_2]}\ol{Z}(t) > \eta\bigr) = 0.
\end{equation}
Then the statement of Lemma~\ref{lemma:boundary-property-of-L} is true. 
\end{lemma}

Consider the subset 
$$
\CG := \{(x, y) \in C[0, T]\times C[0, T]\mid y(t_2) > y(t_1) + \eta,\ \min_{[t_1, t_2]}x(t) > \eta\}.
$$
It is open in $C[0, T]\times C[0, T]$. Suppose we proved that
\begin{equation}
\label{eq:prob-tending-to-zero}
\lim\limits_{k \to \infty}\MP\left(\left(X^C_{n'_k}, L_{n'_k}\right) \in \CG\right) = 0.
\end{equation}
Then by the portmanteau theorem, $\MP\left((\ol{Z}, \ol{L}) \in \CG\right) = 0$, which is equivalent to~\eqref{eq:prob-zero-for-C}. Therefore, it suffices to show~\eqref{eq:prob-tending-to-zero} to prove~\eqref{eq:prob-zero-for-C}. Define $s^{-1}$ to be the inverse function of $s$, with the domain $(s(-\infty), s(\infty))$. Let 
\begin{equation}
\label{eq:def-en}
\eps_n := \max\limits_{x \in [\eta, C]}\left|s^{-1}(s_n(x)) - x\right|.
\end{equation}
Then $\eps_n \to 0$: indeed, $s_n \to s$ uniformly on $[\eta, C]$, and by Lemma~\ref{lemma:aux} we get: $s^{-1}(s_n(x)) \to s^{-1}(s(x)) = x$ uniformly on $[\eta, C]$. Define also 
\begin{equation}
\label{eq:def-of-Yn}
Y_n(t) \equiv s^{-1}(s_n(X^C_n(t))) - Z_n(t\wedge T^{(n)}_C).
\end{equation}
The next two lemmata form the crux of the proof of Lemma~\ref{lemma:boundary-property-of-L}.

\begin{lemma}
\label{lemma:crucial-calculations}
For large enough $n$, the process $Y_n$ is well defined, and for $t < T^{(n)}_C$, we have: 
\begin{equation}
\label{eq:formula-Yn}
\md Y_n(t) = \be_n(X^C_n(t))\md t + \al_n(X^C_n(t))\md W(t),
\end{equation}
where the drift and diffusion coefficients $\be_n(\cdot)$, $\al_n(\cdot)$ are given by
$$
\al_n(x) := \frac{s'_n(x)\si_n(x)}{s'(s^{-1}(s_n(x)))} - \si_n(x),
$$
$$
\be_n(x) = \frac{s'^2_n(x)\si_n^2(x)}{\si^2(s^{-1}(s_n(x)))s'^2(s^{-1}(s_n(x)))}g(s^{-1}(s_n(x))) - g(x).
$$
\end{lemma}

\begin{lemma}
\label{lemma:properties-of-coefficients}
For every interval $[x_1, x_2] \subseteq (0, \infty)$, we have: $\al_n, \be_n \to 0$ uniformly on $[x_1, x_2]$. 
\end{lemma}

From~\eqref{eq:def-of-Yn} and the definition~\eqref{eq:Zn} of $Z_n$, we get:
\begin{equation}
\label{eq:another-form-of-Y}
Y_n(t) \equiv s^{-1}(s_n(X^C_n(t))) - X_n^C(t) + L_n(t).
\end{equation}
From the definition~\eqref{eq:def-en} of $\eps_n$, we get:
\begin{equation}
\label{eq:special-inclusion}
\left\{\min\limits_{[t_1, t_2]}X^C_n(t) > \eta\right\} \subseteq \left\{\max\limits_{t_1 \le t \le t_2}|s^{-1}(s_n(X^C_n(t))) - X_n^C(t)| \le \eps_n\right\}.
\end{equation}
Combining~\eqref{eq:another-form-of-Y} and~\eqref{eq:special-inclusion}, we have:
\begin{equation}
\label{eq:comparison-of-events}
\left\{\min\limits_{[t_1, t_2]}X^C_n(t) > \eta\right\} \subseteq 
\left\{\max\limits_{t_1 \le t \le t_2}|Y_n(t) - L_n(t)| \le \eps_n\right\}.
\end{equation}

\begin{lemma}
\label{lemma:final}
Fix $\eps > 0$. We introduce the following event:
$$
A_n(\eps) := \left\{\exists\ t'_1, t'_2 \in [t_1, t_2] \mid |Y_n(t'_2) - Y_n(t'_1)| \ge \eps,\ \min\limits_{[t_1, t_2]}X^C_n > \eta\right\}.
$$
Then $\MP\left(A_n(\eps)\right) \to 0$ as $n \to \infty$. 
\end{lemma}

Now, we can complete the proof of~\eqref{eq:prob-tending-to-zero}. Assume the following event has happened:
$$
A_n := \left\{\left|L_n(t_2) - L_n(t_1)\right| > \eta,\ \min\limits_{[t_1, t_2]}X^C_n > \eta\right\}.
$$
Since $\eps_n \to 0$, there exists $n_1$ such that for $n \ge n_1$, we have: $\eps_n < \eta/3$. From~\eqref{eq:comparison-of-events}, we get that, for $n \ge n_1$, the following event also has happened:
$$
A'_n := \left\{\left|Y_n(t_2) - Y_n(t_1)\right| > \eta - 2\eps_n > \eta/3,\ \ \min\limits_{[t_1, t_2]}X^C_n > \eta\right\}.
$$
Apply Lemma~\ref{lemma:final} for $\eps = \eta/3$ and get that $\MP(A'_n) \to 0$. Therefore, $\MP(A_n) \to 0$, which completes the proof of~\eqref{eq:prob-tending-to-zero}, as well as the proof of Lemma~\ref{lemma:boundary-property-of-L}. 

\subsubsection{Proof of Lemma~\ref{lemma:prob-zero-for-C}} 

First, let us show $\ol{L}$ cannot increase when $\ol{Z} > 0$. Assume the converse; then there exists a subinterval $[t_1, t_2] \subseteq [0, T]$ such that $\ol{L}(t_2) > \ol{L}(t_1)$, and $\ol{Z}(t) > 0$ for $t \in [t_1, t_2]$. Since the function $\ol{L}$ is continuous, there exist rational $q_1, q_2 \in (t_1, t_2)$ such that $q_1 < q_2$ and $\ol{L}(q_2) > \ol{L}(q_1)$. Find a rational $q > 0$ small enough so that $\ol{L}(q_2) > \ol{L}(q_1) + q$, and $\min_{[t_1, t_2]}\ol{Z} > q$. 
Let 
$$
F(q_1, q_2, q) := \left\{\ol{L}(t_2) > \ol{L}(t_1) + q,\ \min\limits_{[t_1, t_2]}\ol{Z} > q\right\}.
$$
Then 
$$
\left\{\int_0^T\ol{Z}(t)\md\ol{L}(t) > 0\right\} = \bigcup F(q_1, q_2, q),
$$
where the union is taken over all rational $q, q_1, q_2$ such that $0 \le q_1 < q_2 \le T$, and $q > 0$. By the assumption of the lemma, $\MP(F(q_1, q_2, q)) = 0$. Observe that this union is countable to complete the proof. 

\smallskip

Next, let us show $\ol{L}$ is a.s. nondecreasing. It suffices to prove that for every $\de > 0$, for all $t_1, t_2 \in [0, T]$ such that $t_1 < t_2 \le T^{\ol{Z}}_C$, we have: 
\begin{equation}
\label{eq:not-quite-increase}
\ol{L}(t_2) - \ol{L}(t_1) \ge -\de(t_2 - t_1).
\end{equation}
Then, in turn, it suffices to show that for every $t_0 \in (0, T^{\ol{Z}}_C)$, there exists a neighborhood $(t_0 - \eps, t_0 + \eps) \subseteq [0, T]$ such that the property~\eqref{eq:not-quite-increase} is satisfied in this neighborhood. Indeed, if we proved this, then the standard compactness argument gives us~\eqref{eq:not-quite-increase} on the whole $[0, T]$. Assume $\ol{Z}(t_0) > 0$. By continuity of $\ol{Z}$, there exists a neighborhood of $t_0$ in which $\ol{Z} > 0$ and therefore, by Lemma~\ref{lemma:boundary-property-of-L}, the process $\ol{L}$ is constant. Then it obviously satisfies~\eqref{eq:not-quite-increase} in this neighborhood. Now, assume $\ol{Z}(t_0) = 0$. From condition (iii) of Theorem~\ref{thm1}, we have: there exist $\de_0 > 0$ and $n_0$ such that for $x \in (-\de_0, \de_0)$ and $n \ge n_0$, we get: $f_n(x) \ge g(0) - \de$. By continuity of $\ol{Z}$, there exists $\eps > 0$ such that $\ol{Z}(t) < \de_0/2$ for $t \in (t_0 - \eps, t_0+\eps)$. By the uniform convergence 
$X^C_{n'_k} \to \ol{Z}$ on $(t_0 - \eps, t_0 + \eps)$, there exists $k_0$ such that for $k \ge k_1$ we get: 
$$
-\de_0 < X^C_{n'_k}(t) < \de_0, \ \ t \in (t_0 - \eps, t_0 + \eps).
$$
We can find $k_2$ large enough so that for $k \ge k_2$, we have: $n'_k \ge n_0$. Then we have: for $k \ge k_1\vee k_2$, 
$$
f_{n'_k}\left(X^C_{n'_k}(t)\right) - g\left(X^C_{n'_k}(t)\right) > -\de,
$$
and for $t_1, t_2 \in (t_0 - \eps, t_0 + \eps),\ t_1 < t_2$, we get:
$$
\ol{L}_{n'_k}(t_2) - \ol{L}_{n'_k}(t_1) = \int_{t_1}^{t_2}
\left[f_{n'_k}\left(X^C_{n'_k}(t)\right) - g\left(X^C_{n'_k}(t)\right)\right]\md t \ge -\de(t_2 - t_1).
$$
Letting $k \to \infty$, we get~\eqref{eq:not-quite-increase}. The proof is complete.  

\subsubsection{Proof of Lemma~\ref{lemma:crucial-calculations}}

First, let us show that the process $Y_n$ is well defined for sufficiently large $n$: that is, for every $[x_1, x_2] \subseteq (0, \infty)$ there exists $n_0$ such that for $n \ge n_0$, the domain $(s(-\infty), s(\infty))$ of $s^{-1}$ contains $[s_n(x_1), s_n(x_2)]$. Indeed, there exists $\eps > 0$ such that $[s(x_1) - \eps, s(x_2) + \eps] \subseteq (s(-\infty), s(\infty))$, because $s'(x) > 0$ for all $x \in \BR$. Moreover, there exists $n_0$ such that for $n \ge n_0$ we have:
$$
s_n(x_1) > s(x_1) - \eps,\ \ s_n(x_2) < s(x_2) + \eps.
$$
Therefore, for $n \ge n_0$ we have: $[s_n(x_1), s_n(x_2)] 
\subseteq (s(-\infty), s(\infty))$. This proves that $Y_n$ is well defined for $n \ge n_0$. 

\medskip

Let us show the representation~\eqref{eq:formula-Yn} for $\md Y_n(t)$. For all $x \in \BR$ and $n = 1, 2, \ldots$, we have:
\begin{equation}
\label{eq:s'}
s'_n(x) = \exp\left(-2\int_{b_n}^x\frac{f_n(z)}{\si_n^2(z)}\md z\right),\ \ 
s'(x) = \exp\left(-2\int_0^x\frac{g(z)}{\si^2(z)}\md z\right).
\end{equation}
Therefore, for $y \in (s(-\infty), s(\infty))$, 
$$
\left(s^{-1}\right)'(y) = \frac1{s'(s^{-1}(y))} = \exp\left(2\int_{0}^{s^{-1}(y)}\frac{g(z)}{\si^2(z)}\md z\right)
$$
and 
\begin{align*}
\left(s^{-1}\right)''(y) = \exp&\left(2\int_0^{s^{-1}(y)}\frac{g(z)}{\si^2(z)}\md z\right)\cdot 2\frac{g(s^{-1}(y))}{\si^2(s^{-1}(y))}\cdot(s^{-1}(y))' \\ & = \exp\left(4\int_0^{s^{-1}(y)}\frac{g(z)}{\si^2(z)}\md z\right)\cdot 2\frac{g(s^{-1}(y))}{\si^2(s^{-1}(y))}.
\end{align*}
By It\^o's formula, applied to the process $s_n(X^C_n(\cdot))$ from~\eqref{eq:SDE-for-s_n(X_n)}, 
\begin{align*}
\md s^{-1}(s_n(X_n^C(t)))  & = \left(s^{-1}\right)'(s_n(X_n^C(t)))s'_n(X_n^C(t))\si_n(X_n^C(t))\md W(t) \\ & + \frac12\left(s^{-1}\right)''(s_n(X_n^C(t)))\left[s'_n(X_n^C(t))\si_n(X^C_n(t))\right]^2\md t \\ & = \exp\left(2\int_0^{s^{-1}(s_n(X_n^C(t)))}\frac{g(z)}{\si^2(z)}\md z - 2\int_{b_n}^{X_n^C(t)}\frac{f_n(z)}{\si_n^2(z)}\md z\right)\si_n(X_n^C(t))\md W(t) \\ & + \frac12\exp\left(4\int_0^{s^{-1}(s_n(X_n^C(t)))}\frac{g(z)}{\si^2(z)}\md z\right)2\frac{g(s^{-1}(s_n(X_n^C(t))))}{\si^2(s^{-1}(s_n(X_n^C(t))))}\times\\ & \times \si_n^2(X_n^C(t))\exp\left(-4\int_{b_n}^{X_n^C(t)}\frac{f_n(z)}{\si_n^2(z)}\md z\right)\md t
\end{align*}
Comparing with~\eqref{eq:s'}, we get:
\begin{align*}
 \md &s^{-1}(s_n(X_n^C(t)))  = \frac{s'_n(X_n^C(t))\si_n(X_n^C(t))}{s'(s^{-1}(s_n(X_n^C(t))))}\md W(t) \\ & + 
\frac{s'^2_n(X_n^C(t))\si_n^2(X_n^C(t))}{\si^2(s^{-1}(s_n(X_n^C(t))))s'^2(s^{-1}(s_n(X_n^C(t))))}g(s^{-1}(s_n(X_n^C(t))))\md t. 
\end{align*} 
The rest of the proof follows trivially from this latter formula and~\eqref{eq:def-of-Yn}. 

\subsubsection{Proof of Lemma~\ref{lemma:properties-of-coefficients}}

First, let us show that $\al_n \to 0$ uniformly on $[x_1, x_2]$. 
Because $\si_n \to \si$ uniformly on $[x_1, x_2]$, we need only to prove that uniformly on $[x_1, x_2]$,
\begin{equation}
\label{eq:conv-additional}
\frac{s'_n(x)}{s'(s^{-1}(s_n(x)))} \to 1,
\end{equation}
and then to apply Lemma~\ref{lemma:aux} (i). Now, let us show~\eqref{eq:conv-additional}. First, $s_n \to s$ uniformly on $[x_1, x_2]$. Since $s^{-1}$ is continuous in a neighborhood $(s(-\infty), s(\infty))$ of $[s(x_1), s(x_2)]$, 
by Lemma~\ref{lemma:aux} (iii), $s^{-1}(s_n(x)) \to s^{-1}(s(x)) \equiv x$ uniformly on $[x_1, x_2]$. Applying Lemma~\ref{lemma:aux} (iii) again, we get: $s'(s^{-1}(s_n(x))) \to s'(x)$ uniformly on $[x_1, x_2]$. By Lemma~\ref{lemma:aux} (ii), uniformly on $[x_1, x_2]$ we have:
$$
\frac1{ s'(s^{-1}(s_n(x)))} \to \frac1{s'(x)}.
$$
Finally, $s'_n \to s'$ uniformly on $[x_1, x_2]$. Applying Lemma~\ref{lemma:aux} (i) again, we get~\eqref{eq:conv-additional}. This completes the proof that $\al_n \to 0$ uniformly on $[x_1, x_2]$. 

\medskip

Now, let us show $\be_n \to 0$ uniformly on $[x_1, x_2]$. Earlier, we showed $s^{-1}(s_n(x)) \to x$ uniformly on $[x_1, x_2]$. Since $g$ and $\si$ are continuous, by Lemma~\ref{lemma:aux} (iii), uniformly on $[x_1, x_2]$, we have:
$$
g(s^{-1}(s_n(x))) \to g(x),\ \ \si(s^{-1}(s_n(x))) \to \si(x).
$$
By Lemma~\ref{lemma:aux} (ii), uniformly on $[x_1, x_2]$, 
$$
\frac1{\si^2(s^{-1}(s_n(x)))} \to \frac1{\si^2(x)}.
$$
Next, we can represent the coefficient $\be_n(x)$ as
$$
\be_n(x) = \frac{(\al_n(x) + \si_n(x))^2}{\si^2(s^{-1}(s_n(x)))}g(s^{-1}(s_n(x))) - g(x).
$$
Finally, we use the already proven fact that $\al_n \to \al$ uniformly on $[x_1, x_2]$. The rest of the proof consists of a few more applications of Lemma~\ref{lemma:aux}. 

\subsubsection{Proof of Lemma~\ref{lemma:final}}

Assume the event $A_n(\eps)$ has happened. We have: $t'_1 < T^{(n)}_C$, because otherwise $Y_n$ would be constant on $[t'_1, t'_2]$. Without loss of generality, we can assume $t'_2 \le T^{(n)}_C$ (otherwise, just substitute $t'_2\wedge T^{(n)}_C$ for $t'_2$). From Lemma~\ref{lemma:crucial-calculations}, we get:
$$
Y_n(t'_2) - Y_n(t'_1) = \int_{t'_1}^{t'_2}\be_n(X_n(t))\md t + \int_{t'_1}^{t'_2}\al_n(X_n(t))\md W(t).
$$
But $X_n(t) \in [\eta, C]$ for $t \in [t'_1, t'_2]$, and $\be_n \to 0$ uniformly on $[\eta, C]$. Find $n_1$ such that for all $n \ge n_1$ and $x \in [\eta, C]$, we have: $|\be_n(x)| < \eps/(2T)$. Then for $n \ge n_1$, we get:
$$
\left|\int_{t'_1}^{t'_2}\be_n(X_n(t))\md t\right| < T\cdot\frac{\eps}{2T} = \frac{\eps}2,\ \ \mbox{therefore}\ \  
\int_{t'_1}^{t'_2}\al_n(X_n(t))\md W(t) > \eps - \frac{\eps}2 - \frac{\eps}2.
$$
Now, consider the process 
$$
M_n = (M_n(t), 0 \le t \le T),\ \ M_n(t) = \int_0^t\al_n(X_n(u))\md W(u).
$$
We have already proved that for $n \ge n_1$, 
$$
A_n(\eps) \subseteq \{\exists t'_1, t'_2 \in [0, T]: M_n(t'_2) - M_n(t'_1) > \eps/2\}.
$$
%Apply Lemma~\ref{lemma:main-aux} to $\de := T,\ K := \max_{[\eta, C]}|\al_n|$ and get (where $\varkappa$ is defined in~\eqref{eq:kappa}):
%$$
%\MP\left(\exists t'_1, t'_2 \in [0, T]: M_n(t'_2) - M_n(t'_1) > \eps/2\right) \le 
%\varkappa\frac{(C - \eta)\left(\max_{[\eta, C]}|\al_n|\right)^5T^{3/2}}{(\eps/2)^6}T \to 0,
%$$
%because $\al_n \to 0$ uniformly on $[\eta, C]$. The rest is trivial.  
Apply Lemma~\ref{lemma:main-aux} from Appendix to $\de := T,\ K := \max_{[\eta, C]}|\al_n|$ and get: for a constant $\vk > 0$,
$$
\MP\left(\exists t'_1, t'_2 \in [0, T]: M_n(t'_2) - M_n(t'_1) > \eps/2\right) \le 
\varkappa\frac{\left(\max_{[\eta, C]}|\al_n|\right)^4T}{(\eps/2)^2}T \to 0,
$$
because $\al_n \to 0$ uniformly on $[\eta, C]$. The rest is trivial.

\section{Appendix}

\begin{lemma} 
\label{lemma:aux} 

Consider the functions $\phi, \phi_n, \psi, \psi_n : \BR \to \BR$, $n = 1, 2, \ldots$

\medskip

(i) Suppose $\phi_n \to \phi$ and $\psi_n \to \psi$ uniformly on $[x_1, x_2] \subseteq \BR$.  If $\phi$ and $\psi$ are continuous on this interval, then 
$\phi_n\psi_n \to \phi\psi$ uniformly on $[x_1, x_2]$.

\medskip

(ii) If $\phi_n \to \phi$ uniformly on $[x_1, x_2]$, and $\phi$ is continuous and strictly positive on this interval, then $1/\phi_n \to 1/\phi$ uniformly on $[x_1, x_2]$. 

\medskip

(iii) If $ \phi_n \to \phi$ uniformly on $[x_1, x_2]$, and $\psi$ is continuous in a neighborhood of $[\phi(x_1), \phi(x_2)]$, then $\psi(\phi_n) \to \psi(\phi)$ uniformly on $[x_1, x_2]$.

\medskip

(iv) If $\phi_n \to \phi$ uniformly on $[x_1, x_2]$, $\psi_n \to \psi$ uniformly on $[0, x_0+ \eps]$ for some $\eps > 0$, and $\phi(x) \in [0, x_0]$ for $x \in [x_1, x_2]$, then $\psi_n(\phi_n) \to \psi(\phi)$ uniformly on $[x_1, x_2]$. 
\end{lemma}

The proof is trivial and is omitted. 

\begin{lemma}
\label{lemma:aux-sigma} 
If $x_n(t) \to x(t)$ uniformly on $[0, T]$, and $x(t)$ is a continuous nonnegative function on $[0, T]$, then $\si_n(x_n(t)) \to \si(x(t))$ uniformly on $[0, T]$. 
\end{lemma}

\begin{proof}
Fix $\de > 0$. By condition (v) of Theorem~\ref{thm1}, there exists $\eps > 0$ such that for $y \in (-\eps, 0)$, we have: $|\si_n(y) - \si_n(0)| < \de,\ n = 1, 2, \ldots$ Let $C := \max_{[0, T]}x(t)$. Because of uniform convergence $x_n \to x$, there exists $n_0$ such that for $n > n_0$, we have: 
\begin{equation}
\label{1}
x_n(t) \in [-\eps, C + \eps],\ t \in [0, T].
\end{equation}
By condition (iv) of Theorem~\ref{thm1}, there exists $n_1$ such that for $n > n_1$, $y \in [0, C+\eps]$, we have: $|\si_n(y) - \si(y)| < \de$. Therefore, for $y \in [-\eps, 0]$ and $n \ge n_0\vee n_1$, we have: 
\begin{equation}
\label{2}
|\si_n(y) - \si(y)| = |\si_n(y) - \si(0)| \le |\si_n(y) - \si_n(0)| + |\si_n(0) - \si(0)| \le \de + \de = 2\de.
\end{equation}
Because the function $\si$ is continuous on $[-\eps, C + \eps]$, there exists $\eta > 0$ such that for $y_1, y_2 \in [-\eps, C + \eps]$, $|y_1 - y_2| < \eta$, we have: $|\si(y_1) - \si(y_2)| < \de$. Take an $n_2$ such that for $t \in [0, T]$, $n \ge n_2$, $|x_n(t) - x(t)| < \eta$. Therefore,
$$
|\si(x_n(t)) - \si(x(t))| < \de\ \ \mbox{for}\ \ n \ge n_0\vee n_1\vee n_2.
$$
From~\eqref{1} and~\eqref{2}, we get: $|\si_n(x_n(t)) - \si(x(t))| \le 2\de$. Thus, 
$$
\left|\si_n(x_n(t)) - \si(x(t))\right| \le \left|\si_n(x_n(t)) - \si(x_n(t))\right| + \left|\si(x_n(t)) - \si(x(t))\right| \le 2\de + \de = 3\de.
$$
Therefore, for every $\de > 0$ we found $N(\de) := n_0\vee n_1\vee n_2$ such  that for every $t \in [0, T]$, we get: $|\si_n(x_n(t)) - \si(x(t))| \le 3\de$. 
\end{proof}

\begin{lemma}
\label{lemma:main-aux}
Take a real-valued random process $\ga = (\ga(t), 0 \le t \le T)$ such that $\int_0^T\ga^2(t)\md t < \infty$ a.s. Let $W = (W(t), 0 \le t \le T)$ be a standard Brownian motion. Fix constants $x_0 \in \BR$ and $\de \in (0, T], K, D > 0$. Define a random process $X = (X(t), 0 \le t \le T)$ by 
$$
X(t) = x_0 + \int_0^t\ga(s)\md W(s),\ \ \ 0 \le t \le T.
$$
Take an interval $I := [\al, \be] \subseteq \BR$ with $D \le \be - \al$. Assume that the process $X$ has the following property:
for all $t \in [0, T]$, if $X(t) \in I$, then $|\ga(t)| \le K$. Consider the following event:
$$
A := \{\exists t_1, t_2 \in [0, T]:\ 0 \le t_2 - t_1 \le \de,\ X(t_1), X(t_2) \in I,\ |X(t_2) - X(t_1)| \ge D\}.
$$
Then 
$$
\MP(A) \le \varkappa \frac{TK^4}{D^2}\de,\ \ \mbox{for}\ \ \varkappa = \frac{8192}{3}. 
$$
\end{lemma}

\begin{proof} Take $M$ to be large enough so that $\de_0 := T/M \le \de$, for example $M = \lfloor T/\de\rfloor + 1$. Then $T/\de \ge 1$, because $\de \in (0, T]$. Therefore, $M \le 2T/\de$, and $\de_0 \le \de \le 2\de_0$. Now, let $s_i := \de_0i$ for $i = 0, \ldots, M$. These points partition the whole time interval $[0, T]$ into $M$ subintervals of equal length $[s_{i-1}, s_i]$, $i = 1, \ldots, M$. For each $i = 0, \ldots, M-3$, consider the event 
$$
A_{i}  =\left\{\exists t \in [s_i,s_{i+3}]: X(t),\, X(s_i)\in I, |X(t) - X(s_i)| \ge D/2\right\}
$$
By continuity of $X$, we can also write $A_i$ in an equivalent form:
$$
A_i =  \left\{\exists t \in [s_i,s_{i+3}]: X(t),\, X(s_i)\in I, |X(t) - X(s_i)| = D/2\right\}.
$$

\begin{lemma} We have: 
$$
A \subseteq \bigcup\limits_{i=0}^{M-3}A_i.
$$
\label{lemma:union}
\end{lemma}

\begin{proof} Suppose the event $A$ has happened. That is, there exist $t_1, t_2 \in [0, T]$ such that $|t_1 - t_2| \le \de$ and $|X(t_1) - X(t_2)| \ge D$. Then $|t_1 - t_2| \le 2\de_0$. In other words, the two points $t_1$ and $t_2$ are at a distance no more than twice the length of a small subinterval. Therefore, there exists $i = 0, \ldots, M-3$ such that $s_i \le t_1 \le t_2 \le s_{i+3}$ or $s_i \le t_2 \le t_1 \le s_{i+3}$, depending on whether $t_1$ or $t_2$ is greater. Now, we shall prove that the event $A_i$ happened. Assume the converse; then 
$$
|X(t_1) - X(s_i)| < \frac D2\ \ \mbox{and}\ \ |X(t_2) - X(s_i)| < \frac D2.
$$
Therefore, $|X(t_1) - X(t_2)| \le |X(t_1) - X(s_i)| + |X(t_2) - X(s_i)| < \frac D2 + \frac D2 = D$. 
This contradiction completes the proof. 
\end{proof}

\begin{lemma} Assume $X(0) = x \in [\al + \eps, \be - \eps]$ for some $\eps > 0$. 
Let 
$$
\tau := \inf\{t \in [0, T]\mid |X(t) - x| = \eps\}.
$$
Then for $\eta > 0$ we have: 
$$
\MP(\tau \le \eta) \le \frac{256}{27}\frac{K^4\eta^2}{\eps^2}.
$$
\label{lemma:eps-bound}
\end{lemma}

\begin{proof} The process $X(t\wedge\tau)$ is a square-integrable martingale with values in $I = [\al, \be]$. Therefore,  
$$
\langle X\rangle_{t\wedge\tau} = \int_0^{t\wedge\tau}\ga^2(u)\md u \le K^2(t\wedge\tau).
$$
Let us make a time-change: for some standard Brownian motion $B = (B(s), s \ge 0)$, we have: $X(t\wedge\tau) = x + B\left(\langle X\rangle_{t\wedge\tau}\right)$, for all $ t \ge 0$. Therefore, 
\begin{equation}
\label{1201}
\{\tau \le \eta\} = \{|X(\tau\wedge\eta) - x| \ge \eps\} = \left\{|B\left(\langle X\rangle_{\tau\wedge\eta}\right)| \ge \eps\right\} \subseteq \left\{\max\limits_{[0, K^2\eta]}|B(s)| \ge \eps\right\}.
\end{equation}
By Markov's inequality, 
\begin{equation}
\label{1202}
\MP\left\{\max\limits_{[0, K^2\eta]}|B(s)| \ge \eps\right\} \le \frac{\ME\max_{[0, K^2\eta]}B(s)^4}{\eps^4}.
\end{equation}
By Doob's martingale inequality and the fact that for $\xi \backsim \CN(0, \si^2)$ we have: $\ME\xi^4 = 3\si^4$,  
\begin{equation}
\label{1203}
\ME\max\limits_{[0, K^2\eta]}B(s)^4 \le \left(\frac43\right)^4\ME [B(K^2\eta)]^4 = 
\frac{256}{81}\cdot 3(K^2\eta)^2 = \frac{256}{27}K^4\eta^2.
\end{equation}
Comparing~\eqref{1201},~\eqref{1202} and~\eqref{1203}, we complete the proof. 
\end{proof}

\begin{lemma}
\label{eq:probability-of-A0} $\MP(A_0) \le \frac{4096}{3}D^{-2}K^4\de^2$. 
\end{lemma}

\begin{proof} {\it Case 1:} $x = X(0) \in [\al + D/2, \be - D/2]$. Apply Lemma~\ref{lemma:eps-bound} to $\eps = D/2$ and $\eta = s_3 = 3\de_0$. Since $\de_0 \le \de$, we get:
$$
\MP(A_0) = \MP(\tau \le s_3) \le \frac{256}{27}\frac{9K^4\de_0^2}{(D/2)^2}
 = \frac{1024}{3}\frac{K^4\de_0^2}{D^2} \le \frac{1024}{3}\frac{K^4\de^2}{D^2},
$$

\medskip

{\it Case 2:} $x = X(0) \in [\al, \al + D/2]$. Then $X(\tau) = x + D/2$ (because there is not enough ``room'' in the interval $I = [\al, \be]$ for the process $X$ to go down rather than up). 
Let 
$$
\tau' := \inf\{t \ge 0\mid X(t) = x + D/4\}. 
$$
Then $X(\tau') \in [\al + D/4, \be - D/4]$, because $D \le \be - \al$. Apply Lemma~\ref{lemma:eps-bound} to $X(t) - X(\tau')$ instead of $X$, to $\eps = D/4$ and $\eta = s_3$. Then we get:
$$
\MP(A_0) \le \frac{256}{27}\frac{9K^4\de_0^2}{(D/4)^2} \le \frac{4096}{3}\frac{K^4\de^2}{D^2}. 
$$

\medskip

{\it Case 3:} $x = X(0) \in [\be - D/2, \be]$. This case is analogous to Case 2. 
\end{proof}

Similarly, we can show that  for $i = 0, \ldots, M-3$, 
$$
\MP(A_i) \le \frac{4096}{3}\frac{K^4\de^2}{D^2}.
$$
Applying Lemma~\ref{lemma:union}, we get:
$$
\MP(A) \le \SL_{i=0}^{M-3}\MP(A_i) \le M\frac{4096}{3}\frac{K^4\de^2}{D^2}.
$$
But $\de \le T$, and so $M \le 2T/\de$. The rest of the proof is trivial. 
\end{proof}

\section*{Acknoweldgements}

The authors would like to thank \textsc{Ioannis Karatzas}, \textsc{Mykhaylo Shkolnikov}, \textsc{Leszek Slominski}, and \textsc{Ruth Williams} for help and useful discussion. This research was partially supported by NSF grants DMS 1007563, DMS 1308340, DMS 1405210, and DMS 1409434.

%\bibliographystyle{plain}
%
%\bibliography{aggregated}

\end{document}